\documentclass[11pt,english,letterpaper]{smfart}
  
\usepackage{dsfont,eucal,mathptmx,txfonts}

\usepackage[text={6.5in,9.2in},centering]{geometry}

\newcommand{\inv}{^{\raisebox{.2ex}{$\scriptscriptstyle-1$}}}   
   
\usepackage[dvipsnames]{xcolor}   
\usepackage{xparse}
\usepackage{xr-hyper}
\definecolor{brightmaroon}{rgb}{0.76, 0.13, 0.28}
\usepackage[linktocpage=true,colorlinks=true,hyperindex,citecolor=teal,linkcolor=blue]{hyperref}

\newtheorem{thm}{Theorem}[section]
\newtheorem{prop}[thm]{Proposition}
 
\newtheorem{lem}[thm]{Lemma}
\newtheorem{cor}[thm]{Corollary}
\newtheorem{rem}[thm]{Remark}

\author{A. Goswami}
\address{
[1] Department of Mathematics and Applied Mathematics, University of Johannesburg, P.O. Box 524, Auckland Park 2006, South Africa. [2] National Institute for Theoretical and Computational Sciences (NITheCS), South Africa.}
\email{agoswami@uj.ac.za}

\title{Is\'{e}ki spaces of semirings}

\date{}



\begin{document}
\begin{abstract}
The aim of this paper is to study Is\'{e}ki spaces of distinguished classes of ideals of a semiring
endowed with a topology. We show that every Is\'{e}ki space is quasi-compact whenever the semiring is Noetherian. We characterize Is\'{e}ki spaces for which every non-empty irreducible closed subset has a unique generic point. Furthermore, we provide a sufficient condition for the connectedness of Is\'{e}ki spaces and show that the strongly connectedness of an Is\'{e}ki space implies the existence of non-trivial idempotent elements of semirings.
\end{abstract}
\maketitle

\section{Introduction}
 
Since the introduction of semirings by \textsc{Vandiver} \cite{V34}, it is natural to compare and extend results from rings to semirings. One  may think that  semirings can always to be extended to rings, but \textsc{Vandiver} \cite{V39} gave examples of semirings which can not be embedded in rings. Moreover, semirings arise naturally when we consider the set of endomorphisms of a  commutative additive semigroup.

As for rings, various classes of ideals have been topologized with hull-kernel topology (also known as Stone topology or Jacobson topology, or Zariski topology) and the corresponding spaces have been studied intensively. It all started with \textsc{Stone} \cite{S37}, imposing hull-kernel topology on maximal ideals of a Boolean ring.  The same have been done on rings of continuous functions and commutative normed rings by \textsc{Gelfand \& Kolmogoroff}  \cite{GK39} (see also \cite{GJ60,S39}) and by \textsc{Gelfand \& \v{S}ilov} \cite{GS41} respectively, whereas for the algebra of all continuous complex-valued functions, \textsc{Loomis} \cite{L53} considered the same topology on maximal ideals. \textsc{Jacobson} \cite{J45} (see also \cite{J56}) considered that topology on the set of primitive ideals of an arbitrary ring and \textsc{Grothendieck} \cite{G60} used it on prime ideals of a commutative ring to construct affine schemes. The hull-kernel topology on minimal prime ideals of a commutative ring have been studied by \textsc{Henriksen \& Jerison} \cite{HJ65} (see also \textsc{Hochster} \cite{H71}). In a more recent paper, \textsc{Azizi} \cite{A08} endowed it on strongly irreducible ideals of a commutative ring. 

As far as semirings are concerned, \textsc{Is\'{e}ki} \cite{I56} considered various algebraic properties of distinguished classes of ideals and studied hull-kernel topology on strongly irreducible ideals. In \cite{IM56v}, \textsc{Is\'{e}ki \& Miyanaga} studied the same on maximal ideals of a semiring, whereas \textsc{Is\'{e}ki}  \cite{I56v} studied that topology on prime ideals and have called them structure spaces. In the context of $(m,n)$-semirings, \textsc{Hila}, \textsc{Kar}, \textsc{Kuka} \& \textsc{Naka} \cite{HKKN18} studied structure spaces of $n$-ary prime $k$-ideals, $n$-ary prime full $k$-ideals, $n$-ary prime ideals, maximal ideals, and strongly irreducible
ideals.

Note that we can not endow a hull-kernel topology on an arbitrary class of ideals of a ring or of a semiring. \textsc{McKnight} \cite[Section 2.2, p.\,11]{M53} characterized such classes of ideals for rings and the same characterization holds good for semirings. The closest topology which is of the hull-kernel-type is coarse lower topology and it can be imposed on any class of ideals. Moreover, this topology coincide with hull-kernel topology whenever the class of ideals is `good' enough. An attempt has been made in \textsc{Dube \& Goswami} \cite{DG22} to study this spaces (called ideal spaces) of all topologized classes of ideals of a commutative ring.       
 
Although the book \cite{G02} by \textsc{G\l azek} covers references on semirings at the encyclopedic level, but to best of author's knowledge, a study of all classes of topologized  ideals of a semiring has never been considered before. Our purpose of this paper is to generalize the notion of ideal spaces of commutative rings to Is\'{e}ki spaces (in honour of \textsc{Kiyoshi Is\'{e}ki}) of semirings.  We shall use definitions and  results, whenever applicable for semirings, from \textsc{Dube \& Goswami} \cite{DG22} (and also from \textsc{Finocchiaro, Goswami, \& Spirito} \cite{FGS22}) without explicitly referring to them.    
  
\section{Preliminaries}  

Recall from \textsc{Golan} \cite{G99} that a \emph{semiring} is a system $(\mathfrak S,+,0,\cdot, 1)$ such that $(\mathfrak S,+,0)$ is a commutative monoid, $(\mathfrak S, \cdot,1)$ is a monoid, $0r=0=r0$ for all $r\in \mathfrak S,$ and $\cdot$ distributes over $+$ both from the left and
from the right sides.
A semiring $\mathfrak S$ is called \emph{commutative} if $rr'=r'r$ for all $r,r'\in \mathfrak S.$ All our semirings are assumed to be commutative. A \emph{semiring homomorphism} $\phi\colon \mathfrak S\to \mathfrak S'$ is a map such that (i) $\phi(a+b)=\phi(a)+\phi(b),$ (ii) $\phi(ab)=\phi(a)\phi(b),$ and $\phi(1)=1$ for all $a, b\in \mathfrak S.$ A semiring homomorphism $\phi$ is called an \emph{isomorphism} if $\phi$ is also a bijection on the underlying sets.

An \emph{ideal} $\mathfrak{a}$ of a semiring $A$ is a nonempty proper subset of $\mathfrak S$ satisfying the conditions:
(i) $a+b\in \mathfrak{a}$ and
(ii) $ra\in \mathfrak{a}$
for all $a, b\in \mathfrak{a}$ and $r\in \mathfrak S.$ If $\{\mathfrak{a}_{\lambda}\}_{\lambda\in \Lambda}$ is a family of ideals of a semiring $\mathfrak S$, then $\bigcap_{\lambda\in \Lambda} \mathfrak{a}_{\lambda}$ is also an ideal of $\mathfrak S$. The \emph{sum} of a family $\{\mathfrak{a}_{\lambda}\}_{\lambda\in \Lambda}$ of ideals of a semiring $\mathfrak S$ is defined by
\begin{equation}\label{sum}
\sum_{\lambda\in \Lambda} \mathfrak{a}_{\lambda}=\left\{\sum_i^n a_{\lambda_i} \mid a_{\lambda_i} \in \mathfrak{a}_{\lambda_i}, n\in \mathds{N}\right\},
\end{equation} 
which is also an ideal of $\mathfrak S$. If $\mathfrak{a}$ and $\mathfrak{b}$ are two ideals of $\mathfrak S$, then their \emph{product} $\mathfrak{a}\mathfrak{b}$ is the ideal generated by the set $\{a\cdot b\mid a\in \mathfrak{a}, b\in \mathfrak{b}\}.$ As it has been pointed out by \textsc{Brown \& McCoy} \cite{BM58} that it does not matter whether the product $\mathfrak{a}\mathfrak{b}$ of ideals $\mathfrak{a}$ and $\mathfrak{b}$ is defined to be the set of all finite sums $\sum a_{\lambda} b_{\lambda}$ ($a_{\lambda}\in \mathfrak{a}$, $b_{\lambda}\in \mathfrak{b}$), or the smallest ideal of $\mathfrak S$ containing all products $a_{\lambda} b_{\lambda}$, or merely the set of all these products. For rings, \textsc{Behrens} \cite{B56} has used the second
of these definitions, whereas \textsc{Amitsur} \cite{A54} has applied the third. 

The \emph{radical} $\sqrt{\mathfrak{a}}$ of an ideal $\mathfrak{a}$ of a semiring $\mathfrak S$ is defined by 
\begin{equation}\label{rada}
\sqrt{\mathfrak{a}}=\{ r\in \mathfrak S\mid r^n\in \mathfrak{a}\;\text{for some}\; n\in \mathds{N}_{>0}\}.
\end{equation}
It is easy to verify that $\mathfrak{a}\subseteq \sqrt{\mathfrak{a}}$ and $\sqrt{\mathfrak{a}}$ is also an ideal of $\mathfrak S$. An ideal $\mathfrak{a}$ is called a \emph{radical ideal} if $\mathfrak{a}=\sqrt{\mathfrak{a}}.$  An ideal $\mathfrak{p}$ of a semiring $\mathfrak S$ is called \emph{prime} if $ab\in \mathfrak{p}$ implies $a\in \mathfrak{p}$ or $b\in \mathfrak{p}$ for all $a, b \in \mathfrak S.$ Likewise in rings, radicals of semirings also have the following important representation.

\begin{prop}[\cite{N18}]
If $\mathfrak{a}$ is an ideal of a semiring $\mathfrak S$, then 
\begin{equation}
\sqrt{\mathfrak{a}}=\bigcap \left\{ \mathfrak{p}\mid \mathfrak{a}\subseteq \mathfrak{p}\;\text{and}\;\; \mathfrak{p}\;\;\text{is a prime ideal of}\;\; \mathfrak S\right\}.
\end{equation} 
\end{prop}

An ideal $\mathfrak{m}$ of a semiring $\mathfrak S$ is said to be \emph{maximal} if $\mathfrak{m}$ is not properly contained in any other ideals of $\mathfrak S$. 
\textsc{Slowikowski \& Zawadowski} \cite[Theorem 2]{SZ55} proved the following

\begin{prop}\label{maxe}
Every ideal of a semiring $\mathfrak S$ is a subset of at least one maximal ideal of $\mathfrak S$.
\end{prop}

Furthermore, by \textsc{Nasehpour} \cite{N18}, we have the following

\begin{prop}
Every maximal ideal of a semiring is a prime ideal.
\end{prop}

A proper ideal $\mathfrak{a}$ of a semiring $\mathfrak S$ is called \emph{primary} (see \textsc{Golan} \cite{G99}) if $xy\in \mathfrak{a}$
implies either
$x\in \mathfrak{a}$ or
$y^n\in \mathfrak{a}$
for some $n\in \mathds{N}.$ By \textsc{Is\'{e}ki} \cite{I56}, an ideal $\mathfrak{s}$ is called \emph{strongly irreducible} (\emph{irreducible}), if 
$\mathfrak{a}\cap \mathfrak{b}\subseteq \mathfrak{s}$ ($\mathfrak{a}\cap \mathfrak{b}= \mathfrak{s}$) for any two ideals $\mathfrak{a},$ $\mathfrak{b}$ implies $\mathfrak{a}\subseteq \mathfrak{s}$ ($\mathfrak{a}=\mathfrak{s}$) or $ \mathfrak{b}\subseteq \mathfrak{s}$ ($\mathfrak{b}= \mathfrak{s}$). 
Any one of the classes of ideals of a semiring $\mathfrak S$ considered above will be called a \emph{spectrum} of $\mathfrak S$ and we shall denote it by $\sigma_{\mathfrak S}$. We shall assume that $\mathfrak S\notin \sigma_{\mathfrak S}$ for all spectra of $\mathfrak S$. 

\section{Is\'{e}ki spaces}

Suppose $\mathfrak{S}$ is a semiring. The ideal topology (also known as coarse lower topology \textsc{Dickmann}, \textsc{Schwartz}, \& \textsc{Tressal} \cite[A.8, p.\,589]{DST19} or lower topology \textsc{Gierz} \textit{et.\;\!al.} \cite[Definition O-5.4, p.\,43]{G et. al.}) on a spectrum $\sigma_{\mathfrak S}$ of $\mathfrak S$ will be the topology for which the sets of the type  
\begin{equation} \label{clop}
\{\mathfrak a \}^{\uparrow} =\{\mathfrak x\in \sigma_{\mathfrak S}\mid \mathfrak a\subseteq \mathfrak x \} 
\end{equation}

(where $\mathfrak a$ runs among the ideals of $\mathfrak S$) form a subbasis of closed sets. A spectrum $\sigma_{\mathfrak S}$ of  a semiring $\mathfrak S$ endowed with an ideal topology is called an \emph{Is\'{e}ki space}, and we denote the Is\'{e}ki  space also by $\sigma_{\mathfrak S}$.
The following results are easy to prove.

\begin{prop}\label{upp}
The subbasic closed subsets of an Is\'{e}ki space have the following properties.
\begin{enumerate}
\item The map $^{\uparrow}\colon \mathrm{Ideal}(\mathfrak{S})\to \mathcal{P}(\sigma_{\mathfrak{S}})$ defined in (\ref{clop}) is order-reversing, where $\mathrm{Ideal}(\mathfrak{S})$ denotes the set of all ideals of the semiring $\mathfrak{S}$ and $\mathcal{P}(\sigma_{\mathfrak{S}})$ is the power set of the spectrum $\sigma_{\mathfrak{S}}$. Furthermore, $\mathfrak{o}^{\uparrow}=\mathfrak{S}$ and $\mathfrak{S}^{\uparrow}=\emptyset$, where $\mathfrak{o}$ is the zero ideal of $\mathfrak{S}.$

\item \label{ihs} For any two ideals $\mathfrak{a},$ $\mathfrak{b}$ of $\mathfrak{S}$, $\mathfrak{a}^{\uparrow}\cup \mathfrak{b}^{\uparrow} \subseteq (\mathfrak{a}\cap \mathfrak{b})^{\uparrow}\subseteq (\mathfrak{ab})^{\uparrow}.$

\item\label{insum} For a family $\{\mathfrak{a}_{\alpha}^{\uparrow} \}_{\alpha \in \Lambda}$ of subbasic closed sets, $\bigcap_{\alpha \in \Lambda}\mathfrak{a}_{\alpha}^{\uparrow}=\left( \sum_{\alpha \in \Lambda}\mathfrak{a}_{\alpha} \right)^{\uparrow}$.

\item For every ideal $\mathfrak{a}$ of $\mathfrak{S},$ $\mathfrak{a}^{\uparrow}\supseteq(\sqrt{\mathfrak{a}})^{\uparrow},$ where $\sqrt{\mathfrak{a}}$ is as defined in (\ref{rada}).

\item A spectrum $\sigma_{\mathfrak{S}}$ is a radical ideal if and only if $\mathfrak{a}^{\uparrow}=(\sqrt{\mathfrak{a}})^{\uparrow}$ for every ideal $\mathfrak{a}$ in $\mathfrak{S}.$
\end{enumerate}
\end{prop}

It is evident from (\ref{clop}) that if $\mathfrak{x}\neq \mathfrak{x}'$ for any two $\mathfrak{x}, \mathfrak{x}'\in \sigma_{\mathfrak S}$, then $\mathfrak{x}^{\uparrow}\neq \mathfrak{x}'^{\uparrow}.$ Therefore, we have

\begin{prop}\label{t0a}
Every Is\'{e}ki space $\sigma_{\mathfrak S}$ is a $T_0$-space. 
\end{prop} 

Recall that a topological space is called \emph{quasi-compact} if every open cover of it has a finite subcover, or equivalently, the space satisfies the finite intersection property. In this definition of quasi-compactness, we do not assume the space is $T_2.$ Quasi-compactness of prime, maximal, and strongly irreducible ideals endowed with hull-kernel topology (= coarse lower topology) have respectively been shown in \textsc{Is\'{e}ki} \cite{I56v}, \textsc{Is\'{e}ki \& Miyanaga} \cite{IM56v}, and \textsc{Is\'{e}ki} \cite{I56}. The following result gives a sufficient condition for quasi-compactness of all Is\'{e}ki spaces. The proof relies on the application of Alexander subbase theorem. 

\begin{thm}\label{csb}
If a semiring $\mathfrak S$ has an identity and if a spectrum $\sigma_{\mathfrak S}$ of $\mathfrak S$ contains all maximal ideals of $\mathfrak S$, then $\sigma_{\mathfrak S}$ is a quasi-compact space. 
\end{thm}

\begin{proof}
Suppose  $\{\mathcal K_{ \lambda}\}_{\lambda \in \Lambda}$ is a family of subbasic closed sets of $\sigma_{\mathfrak S}$  with $\bigcap_{\lambda\in \Lambda}\mathcal K_{ \lambda}=\emptyset.$ This implies  
$\mathcal K_{ \lambda}=\mathfrak{x}_{\lambda}^{\uparrow}$ for some ideals $\mathfrak{x}_{\lambda}$ of $\mathfrak S$, and 
$$\bigcap_{\lambda \in \Lambda}\mathfrak{x}_{\lambda}^{\uparrow}=\left(\sum_{\lambda \in \Lambda} \mathfrak{x}_{\lambda}  \right)^{\uparrow}=\emptyset, $$
where the sum is defined in (\ref{sum}), and the first equality follows from Proposition \ref{upp} (\ref{insum}).
If $\sum_{\lambda \in \Lambda} \mathfrak{x}_{\lambda}  \neq \mathfrak S,$ then we must have a maximal ideal $\mathfrak{m}$ of $\mathfrak S$ (see Proposition \ref{maxe}) such that $\sum_{\lambda \in \Lambda} \mathfrak{x}_{\lambda}  \subseteq \mathfrak{m}.$ Moreover, 
$$ \mathfrak{x}_{\lambda}  \subseteq \sum_{\lambda \in \Lambda} \mathfrak{x}_{\lambda}  \subseteq \mathfrak{m},$$
for all $\lambda \in \Lambda.$ Therefore $\mathfrak{m}\in \mathfrak{x}_{\lambda}^{\uparrow}=\mathcal K_{\lambda}$ for all $\lambda \in \Lambda$, a contradiction of our assumption. Hence $\sum_{\lambda \in \Lambda} \mathfrak{x}_{\lambda}=\mathfrak S,$ and the identity $1\in \sum_{\lambda \in \Lambda} \mathfrak{x}_{\lambda}.$ This implies the existence of a finite subset $\{\lambda_{\scriptscriptstyle 1}, \ldots, \lambda_{\scriptscriptstyle n}\}$ of $\Lambda$ such that $1= \sum_{i=1}^n x_{\lambda_i}$ (where $x_{\lambda_i}\in \mathfrak{x}_{\lambda_i}$), and hence  $R= \sum_{i=1}^n \mathfrak{x}_{\lambda_i},$ which establishes the finite intersection property. Therefore, $\sigma_{\mathfrak S}$ is quasi-compact by Alexander's subbase theorem.
\end{proof} 

\begin{cor}
The Is\'{e}ki spaces of maximal, prime, strongly irreducible, primary, irreducible, radical ideals are all quasi-compact.
\end{cor}

\begin{rem}
\emph{In a $(m,n)$-semiring, our coarse lower topology coincide with the hull-kernel topology on $n$-ary prime full $k$-ideals, $n$-ary prime ideals, and strongly irreducible
ideals, and hence, we obtain quasi-compactness (see respectively \cite[Theorem 3.1, Theorem 5.4, and Theorem 6.2]{HKKN18}).}
\end{rem}

Notice that in Theorem
\ref{csb}
the containment of all maximal ideals to an arbitrary spectrum $\sigma_{\mathfrak S}$ is a sufficient condition for quasi-compactness of the Is\'{e}ki space $\sigma_{\mathfrak S}$. For the spectrum of finitely generated ideals of a semiring, for instants, it is also a necessary condition.

\begin{prop}
If the Is\'{e}ki space $\sigma_{\mathfrak S}$ of finitely generated proper ideals is quasi-compact then $\sigma_{\mathfrak S}$ contains all maximal ideals of $\mathfrak S.$ 
\end{prop}
 
\begin{proof}
Suppose $\mathfrak m$ is a maximal ideal of the semiring $\mathfrak S$ such that $\mathfrak m$ is not finitely generated. Let us consider the collection of closed subspaces: $$\mathcal C=\left\{\langle x\rangle^\uparrow\bigcap \sigma_{\mathfrak S}\mid x\in \mathfrak m \right\}.$$  We claim that $\bigcap \mathcal C=\emptyset$. If not, let $\mathfrak y\in \bigcap \mathcal C$. Then $\mathfrak b$ is finitely generated and $\mathfrak m\subseteq \mathfrak y$. Since $\mathfrak m$ is not a finitely generated ideal, we must have $\mathfrak y\supsetneq \mathfrak m,$ and that implies $\mathfrak y=\mathfrak S$, which contradicts the fact that $\sigma_{\mathfrak S}$ consists of proper ideals. But clearly $\mathcal C$ has the finite intersection property and hence $\sigma_{\mathfrak S}$ is not quasi-compact.
\end{proof}

If $\mathfrak{S}$ is a Noetherian semiring, then we have the following result which may considered as a semiring-theoretic version of the corresponding result proved in \textsc{Finocchiaro}, \textsc{Goswami}, \& \textsc{Spirito},  \cite{FGS22} for rings. 
 
\begin{thm}
If $\mathfrak S$ is a Noetherian semiring then every Is\'{e}ki space $\sigma_{\mathfrak{S}}$ is quasi-compact.
\end{thm}   

\begin{proof}
Consider a collection $\{\sigma_{\mathfrak{S}}\cap \mathfrak x_{\lambda}^{\uparrow} \}_{\lambda\in \Omega}$ of subbasic closed sets of $\sigma_{\mathfrak{S}}$ with the finite intersection property. By assumption, the ideal $\mathfrak  y=\sum_{\lambda\in \Omega}\mathfrak x_{\lambda}$ is finitely generated, say $\mathfrak y=(\alpha_1,\ldots,\alpha_n)$. For  every $1\leqslant j\leqslant n$, there exists a finite subset $\Lambda_j$ of $\Omega$ such that $\alpha_j\in \sum_{\lambda\in \Lambda_j}\mathfrak x_{\lambda}$. Thus, if  $\Lambda:=\bigcup_{j=1}^n\Lambda_j$, it immediately follows that $\mathfrak y=\sum_{\lambda\in \Lambda}\mathfrak x_{\lambda}$. Hence we have
\begin{align*} 
\bigcap_{\lambda\in \Omega}\left(\sigma_{\mathfrak{S}}\cap \mathfrak x_{\lambda}^{\uparrow}\right)&=\sigma_{\mathfrak{S}}\cap \mathfrak y^{\uparrow}\\&=\sigma_{\mathfrak{S}}\bigcap  \left(\sum_{\lambda\in \Lambda}\mathfrak x_{\lambda} \right)^{\uparrow}\\&= \bigcap_{\lambda\in \Lambda}\left(\sigma_{\mathfrak{S}}\cap \mathfrak x_{\lambda}^{\uparrow}\right) \neq \emptyset,
\end{align*}
since $\Lambda$ is finite and $\{\sigma_{\mathfrak{S}}\cap \mathfrak x_{\lambda}^{\uparrow} \}_{\lambda\in \Omega}$ has the finite intersection property. Then the conclusion follows by the Alexander subbasis Theorem. 
\end{proof}

\begin{cor}
If $\mathfrak S$ is a Noetherian semiring then every Is\'{e}ki space $\sigma_{\mathfrak{S}}$ is  Noetherian.
\end{cor} 
 
Next we consider the separation axiom of sobriety. Recall that a topological space $X$ is called \emph{sober} if every non-empty irreducible closed subset $\mathcal K$ of $X$ is of the form: $\mathcal K=\mathrm{Cl}(\{x\})$, the closure of an unique singleton set $\{x\}$. The following result characterizes sober Is\'{e}ki spaces of a semiring.

\begin{thm}
\label{sob}  
An Is\'{e}ki space $\sigma_{\mathfrak S}$ is sober if and only if
\begin{equation}\label{soc}
\mathfrak{a}^{\uparrow}\supseteq \bigcap_{\substack{\mathfrak{a}\subseteq \mathfrak{x} \\ \mathfrak{x}\in \sigma_{\mathfrak S}}}  \mathfrak{x}^{\uparrow}
\end{equation}
for every non-empty irreducible subbasic closed subset of $\mathfrak{a}^{\uparrow}$ of $\sigma_{\mathfrak S}$.
\end{thm}

\begin{proof}
If $\sigma_{\mathfrak S}$ is a sober space and $\mathfrak{a}^{\uparrow}$ is a non-empty irreducible subbasic closed subset of $\sigma_{\mathfrak S}$, then  
$\mathfrak{a}^{\uparrow}=\mathrm{Cl}(\{\mathfrak{b}\})=\mathfrak{b}^{\uparrow}$
for some $\mathfrak{b}\in \sigma_{\mathfrak S}$, and we have $$\displaystyle\mathfrak{b}=\bigcap_{\substack{\mathfrak{a}\subseteq \mathfrak{x} \\ \mathfrak{x}\in \sigma_{\mathfrak S}}}  \mathfrak{x}^{\uparrow}\in \sigma_{\mathfrak S}.$$  
Conversely, suppose the condition (\ref{soc}) holds for every non-empty irreducible subset of $\sigma_{\mathfrak S}$. Let $\mathcal K$ be an irreducible closed subset of $\sigma_{\mathfrak S}$. Then $$\mathcal K=\bigcap_{i\in \Omega}\left (  \bigcup_{j=1}^m \mathfrak{x}_{ji}^{\uparrow} \right)$$ for some ideals $\mathfrak x_{ji}$ of $\mathfrak S$. Since $\mathcal K$ is irreducible, for every $i\in \Omega$ there exists an ideal $\mathfrak x_{ji}$ of $\mathfrak S$ such that $\displaystyle\mathcal K\subseteq \mathfrak{x}_{ji}^{\uparrow}\subseteq \bigcup_{j=1}^m\mathfrak{x}_{ji}^{\uparrow}$ and thus, if $\mathfrak y= \sum_{i\in \Omega}\mathfrak x_{ji}$, then we have  $$\mathcal K=\bigcap_{i\in \Omega}\mathfrak{x}_{ji}^{\uparrow}=\mathfrak{y}^{\uparrow}=\left(\bigcap_{\substack{\mathfrak{y}\subseteq \mathfrak{s} \\ \mathfrak{s}\in \sigma_{\mathfrak S}}}  \mathfrak{s} \right)^{\uparrow}.$$ By assumption, $\displaystyle\bigcap_{\substack{\mathfrak{y}\subseteq \mathfrak{s} \\ \mathfrak{s}\in \sigma_{\mathfrak S}}}  \mathfrak{s}\in \sigma_{\mathfrak S},$ and thus $\displaystyle\mathcal K=\mathrm{Cl}\left({\bigcap_{\substack{\mathfrak{y}\subseteq \mathfrak{s} \\ \mathfrak{s}\in \sigma_{\mathfrak S}}}  \mathfrak{s}}\right).$  
The uniqueness part follows from Proposition \ref{t0a}.
\end{proof}

\begin{cor}
The Is\'{e}ki spaces of proper ideals, prime ideals, and strongly irreducible ideals are sober.
\end{cor}

Like Alexander subbase theorem, there is no characterization of connectedness in terms of subbasic closed sets. Nevertheless, we wish to present a disconnectivity result of Is\'{e}ki spaces of a semiring that bears resemblance to the fact that if spectrum of prime ideals (of a commutative ring with identity) endowed with Zariski topology is disconnected, then the ring has a proper idempotent element (see \textsc{Bourbaki} \cite{B72}, Section 4.3, Corollary 2). 

We say a closed subbase $\mathcal{S}$ of a topological space $X$ \emph{strongly disconnects} $X$  if there exist two non-empty subsets $A,$ $B$ of $\mathcal{S}$ such that $X=A\cup B$ and $A\cap B=\emptyset$. 
It is clear that if some closed subbase strongly disconnects  a  topological space, then the space is disconnected. Also, if a space is disconnected, then some closed subbase (for instance the collection of all its closed subspaces) strongly disconnects it. 
 
\begin{prop}\label{pr1}  
 Suppose $\mathfrak S$ is a semiring with multiplicative identity $1$ and $\mathfrak S$ has zero Jacobson radical. Let $\sigma_{\mathfrak S}$ be a spectrum of $\mathfrak S$ containing all maximal ideals of $\mathfrak S.$ If the subbase $\mathcal{S}$ of the Is\'{e}ki space $\sigma_{\mathfrak S}$ strongly  disconnects  $\sigma_{\mathfrak S}$, then $\mathfrak S$ has a non-trivial idempotent element.
\end{prop} 

\begin{proof}
Let $\mathfrak{x}$ and $\mathfrak{y}$ be ideals of $\mathfrak S$ such that
\begin{enumerate}
\item $
\mathfrak{x}^{\uparrow}\cap\mathfrak{y}^{\uparrow} =\emptyset,$
\item $\mathfrak{x}^{\uparrow}\cup\mathfrak{y}^{\uparrow} =\sigma_{\mathfrak S},$ and  
\item \label{nnep} $\mathfrak{x}^{\uparrow}\ne\emptyset,  \mathfrak{y}^{\uparrow} \neq \emptyset. $
\end{enumerate}
Since $\mathfrak{x}^{\uparrow}\cap\mathfrak{b}^{\uparrow}=(\mathfrak{x} +\mathfrak{y})^{\uparrow}$, we therefore have $(\mathfrak{x}+\mathfrak{y})^{\uparrow}=\emptyset$ and hence $\mathfrak{x}+\mathfrak{y} =\mathfrak S$ because $\sigma_{\mathfrak S}$ contains all maximal ideals of $\mathfrak S$. On the other hand, 
\[
(\mathfrak{x}\mathfrak{y})^{\uparrow}\supseteq \mathfrak{x}^{\uparrow}\cup\mathfrak{y}^{\uparrow}=\sigma_{\mathfrak S},
\]
which then implies that $\mathfrak{x}\mathfrak{y}$ is contained in every maximal ideal of $\mathfrak S$, and is therefore the zero ideal since $\mathfrak{S}$ has zero Jacobson radical. Note that the condition (\ref{nnep}) implies that neither $\mathfrak{x}$ nor $\mathfrak{y}$ is the entire semiring $\mathfrak S$. So the 
equality $\mathfrak{x}+\mathfrak{y}=\mathfrak S$ furnishes non-zero elements $x\in\mathfrak{x}$ and $y\in\mathfrak{y}$ such that $x+y=1$. Since  $xy=0$ as $xy\in\mathfrak{x}\mathfrak{y}=\mathfrak{o}$, we therefore have
$
x= x(x+y)=x^{ 2}+xy =x^{ 2},
$ 
showing that $x$ is a non-zero  idempotent element in $\mathfrak S$. Since $\mathfrak{x}\ne \mathfrak S$,  $x\ne 1$, and hence $x$ is a non-trivial idempotent element of $\mathfrak S$.
\end{proof}

Theorem \ref{conn} below gives a sufficient condition for an Is\'{e}ki space to be connected and to prove that result we first need the following 

\begin{lem}\label{irrc}
In every Is\'{e}ki space $\sigma_{\mathfrak S}$, the subbasic closed sets of the form: $\{\mathfrak{y}^{\uparrow}\mid \mathfrak{y}\in \mathfrak{y}^{\uparrow}\}$ are irreducible. 
\end{lem}

\begin{proof}
It is sufficient to show that $\mathfrak{y}^{\uparrow}=\mathrm{Cl}(\mathfrak{y})$ whenever $\mathfrak{y}\in\mathfrak{y}^{\uparrow}$. Since $\mathrm{Cl}(\mathfrak{y})$ is the smallest closed set containing $\mathfrak{x}$, and since $\mathfrak{y}^{\uparrow}$ is a closed set containing $\mathfrak{y}$, obviously then  $\mathrm{Cl}(\mathfrak{y})\subseteq \mathfrak{y}^{\uparrow}$. 
For the reverse inclusion, if $\mathrm{Cl}(\mathfrak{y})= \sigma_{\mathfrak S}$, then 
\[ 
\sigma_{\mathfrak S}=\mathrm{Cl}(\mathfrak{y})\subseteq \mathfrak{y}^{\uparrow}\subseteq \sigma_{\mathfrak S}.
\] 
This proves that $\mathfrak{y}^{\uparrow}=\mathrm{Cl}(\mathfrak{y})$. Suppose that $\mathrm{Cl}(\mathfrak{y})\neq \sigma_{\mathfrak S}$. Since $\mathrm{Cl}(\mathfrak{y})$ is a closed set,  there exists an  index set, $\Omega$, such that,  for each $\lambda\in\Omega$, there is a positive integer $x_{\lambda}$ and ideals $\mathfrak{x}_{\lambda 1},\dots, \mathfrak{x}_{\lambda n_\lambda}$ of $\mathfrak S$ such that 
$$ 
\mathrm{Cl}(\mathfrak{y})={\bigcap_{\lambda\in\Omega}}\left({\bigcup_{ i\,=1}^{ x_\lambda}}\mathfrak{x}_{\lambda i}^{\uparrow}\right).
$$
Since  
$\mathrm{Cl}(\mathfrak{y})\neq \sigma_{\mathfrak S},$ we can assume that ${\bigcup_{ i\,=1}^{ x_\lambda}}\mathfrak{x}_{\lambda i}^{\uparrow}$ is non-empty for each $\lambda$. Therefore, $\mathfrak{y}\in   {\bigcup_{ i\,=1}^{ x_\lambda}}\mathfrak{x}_{\lambda i}^{\uparrow}$ for each $\lambda$, and hence $\mathfrak{y}^{\uparrow}\subseteq {\bigcup_{ i=1}^{ x_\lambda}}\mathfrak{x}_{\lambda i}^{\uparrow}$, that is $\mathfrak{y}^{\uparrow}\subseteq \mathrm{Cl}(\mathfrak{y})$ as desired. 
\end{proof}

\begin{thm}\label{conn}
If a spectrum $\sigma_{\mathfrak S}$ of a semiring $\mathfrak S$ contains the zero ideal, then the Is\'{e}ki space $\sigma_{\mathfrak S}$ is connected. 
\end{thm}

\begin{proof}
Since  $\sigma_{\mathfrak S}=\mathfrak{o}^{\uparrow}$ and irreducibility implies connectedness, the claim now follows from Lemma \ref{irrc}.
\end{proof}

\begin{cor}
Is\'{e}ki spaces of proper, finitely generated, principal ideals of a semiring are connected. 
\end{cor}

\textsc{Is\'{e}ki \& Miyanaga} \cite[Theorem 2]{IM56v} has shown that the spectrum of maximal ideals with the Stone topology (= coarse lower topology) is a $T_1$-space. Furthermore, \textsc{Is\'{e}ki} \cite[Theorem 3]{I56v} proved that every prime ideal is maximal, if and only if the
topological space
is a $T_1$-space. Surprisingly, Is\'{e}ki space of maximal ideals of a semiring also characterizes $T_1$-space as we shall see in the following

\begin{thm}
An Is\'{e}ki space $\sigma_{\mathfrak{S}}$ is $T_1$ if and only if the spectrum $\sigma_{\mathfrak{S}}$ is the set of all maximal ideals of $\mathfrak{S}.$ 
\end{thm}   

\begin{proof} 
Suppose $\sigma_{\mathfrak{S}}$ is a $T_1$-space and let $\mathfrak{a}\in \sigma_{\mathfrak{S}}$. Then $\mathfrak{a}\in\mathfrak{a}^{\uparrow}$, and so, by Theorem \ref{irrc}, $\mathrm{Cl}({\{\mathfrak{a}\}})=\mathfrak{a}^{\uparrow}$. Let $\mathfrak{m}$ be a maximal ideal of $\mathfrak{S}$ with $\mathfrak{a}\subseteq\mathfrak{m}$. Then   $$\mathfrak{m}\in 	\mathfrak{a}^{\uparrow}=\mathrm{Cl}({\{\mathfrak{a}\}}) = \{\mathfrak{a}\},$$ where the last equality follows from the hypothesis. Therefore $\mathfrak{m}=\mathfrak{a}$, showing that $\sigma_{\mathfrak{S}}$ is contained in the set of all maximal ideals of $\mathfrak{S}$.  
Conversely, $\mathfrak{m}^{\uparrow}=\{\mathfrak{m}\}$ for every maximal ideal $\mathfrak{m}$ of $\mathfrak{S},$ so that $\mathfrak{m}\in \mathfrak{m}^{\uparrow}$, and hence, by Theorem \ref{irrc}, $\mathrm{Cl}({\{\mathfrak{m}\}})=\{\mathfrak{m}\}$, showing that the Is\'{e}ki space $\sigma_{\mathfrak{S}}$ is a $T_{ 1}$-space.
\end{proof} 

\begin{cor}
Let $\mathfrak{S}$ be a Noetherian ring. If $\sigma_{\mathfrak{S}}$ is a discrete space then $\mathfrak{S}$ is Artinian.  
\end{cor} 

\begin{rem}
\emph{Note that in an $(m,n)$-semiring, maximal ideals endowed with two special topologies generated respectively by $\Delta_x$ and $\Omega_x$ (see \cite[Theorem 4.2 and Theorem 4.3]{HKKN18})  satisfy $T_2$ and $T_1$ separation axioms}.
\end{rem}
 
We now discuss about continuous maps between Is\'{e}ki spaces of semirings. Observe that although inverse image of an ideal under a semiring homomorphism is an ideal, but the same may not hold for an arbitrary spectra $\sigma_{\mathfrak S}$. To resolve this problem, we need to impose that property on a spectrum. 
We say a spectrum $\sigma_{\mathfrak S}$ satisfies the \emph{contraction} property if for any semiring homomorphism $\phi\colon \mathfrak S\to \mathfrak S',$ the inverse image  $\phi\inv(\mathfrak{x}')$ is in $\sigma_{\mathfrak S}$, whenever $\mathfrak{x}'$ is in $\sigma_{\mathfrak S'}.$   
Since the sets $\{\mathfrak{x}^{\uparrow}\mid\mathfrak{x}\;\text{is an ideal of}\; \mathfrak{S}\}$ only form a (closed) subbasis, all our arguments need to be at this level rather than just closed sets.

\begin{prop}\label{conmap}
Let $\sigma_{\mathfrak S}$ be a spectrum satisfying the contraction property. Let $\phi\colon \mathfrak S\to \mathfrak S'$ be a semiring homomorphism  and $\mathfrak{x}'\in\sigma_{\mathfrak S'}.$ 
\begin{enumerate}

\item\label{contxr} The induced map $\phi_*\colon  \sigma_{\mathfrak S'}\to \sigma_{\mathfrak S}$ defined by  $\phi_*(\mathfrak{x}')=\phi\inv(\mathfrak{x}')$ is    continuous.

\item If $\phi$ is  surjective, then the Is\'{e}ki space $\sigma_{\mathfrak S'}$ is homeomorphic to the closed subspace $\mathrm{Ker}(\phi)^{\uparrow}$ of the Is\'{e}ki space $\sigma_{\mathfrak S}.$

\item\label{den} The subset  $\phi_*(\sigma_{\mathfrak S'})$ is dense in $\sigma_{\mathfrak S}$ if and only if $\mathrm{Ker}(\phi)\subseteq \bigcap_{\mathfrak{s}\in \sigma_{\mathfrak S}}\mathfrak{s}.$ 
\end{enumerate}
\end{prop}
   
\begin{proof}      
To show (1), let $\mathfrak{x}$ be an ideal of $\mathfrak{S}$ and $\mathfrak{x}^{\uparrow}$ be a   subbasic closed set of the ideal  space $\sigma_{\mathfrak S}.$ Then  
\begin{align*}
(\phi_*)\inv(\mathfrak{x}^{\uparrow}) &=\{ \mathfrak{x}'\in  \sigma_{\mathfrak S'}\mid \phi\inv(\mathfrak{x}')\in \mathfrak{x}^{\uparrow}\}\\&=\{\mathfrak{x}'\in \sigma_{\mathfrak S'}\mid \phi(\mathfrak{x})\subseteq \mathfrak{x}'\}\\&=\langle\phi(\mathfrak{x}^{\uparrow}\rangle), 
\end{align*} 
and hence the map $\phi_*$  continuous.     
For (2), observe that $\mathrm{Ker}(\phi)\subseteq \phi\inv(\mathfrak{x}')$ follows from the fact that  $\mathfrak{o}\subseteq \mathfrak{x}'$ for all $\mathfrak{x}'\in \sigma_{\mathfrak S'}.$ It can thus been seen that $\phi_*(\mathfrak{x}')\in \mathrm{Ker}(\phi)^{\uparrow},$ and hence $\mathrm{Im}(\phi_*)=\mathrm{Ker}(\phi)^{\uparrow}.$  
If $\mathfrak{x}'\in \sigma_{\mathfrak S'},$ then
$\phi(\phi_*(\mathfrak{x}'))=\phi(\phi\inv(\mathfrak{x}'))=\mathfrak{x}'.$
Thus $\phi_*$ is injective. To show that $\phi_*$ is closed, first we observe that for any   subbasic closed subset  $\mathfrak{a}^{\uparrow}$ of  $\sigma_{\mathfrak S'}$, we have
\begin{align*}
\phi_*(\mathfrak{x}^{\uparrow})&=  \phi\inv(\mathfrak{x}^{\uparrow})\\&=\phi\inv\{ \mathfrak{x}'\in \sigma_{\mathfrak S'}\mid \mathfrak{x}\subseteq   \mathfrak{x}'\}\\&=\phi\inv(\mathfrak{x})^{\uparrow}. 
\end{align*}
Now if $\mathcal K$ is a closed subset of $\sigma_{\mathfrak S'}$ and $\mathcal K=\bigcap_{ \lambda \in \Omega} (\bigcup_{ i \,= 1}^{ n_{\lambda}} \mathfrak{x}_{ i\lambda}^{\uparrow}),$ then
\begin{align*}
\phi_*(\mathcal K)&=\phi\inv \left(\bigcap_{ \lambda \in \Omega} \left(\bigcup_{ i = 1}^{ n_{\lambda}} \mathfrak{x}_{ i\lambda}^{\uparrow}\right)\right)\\&=\bigcap_{ \lambda \in \Omega} \bigcup_{ i = 1}^{x_{\lambda}} \phi\inv(\mathfrak{x}_{ i\lambda}^{\uparrow})
\end{align*}
a closed subset of  $\sigma_{\mathfrak S}.$ Since by (\ref{contxr}), $\phi_*$ is continuous, we have the desired claim.
Finally to prove (3), first we wish to show: $\mathrm{Cl}(\phi_*(\mathfrak{x}'^{\uparrow}))=\phi\inv(\mathfrak{x}'^{\uparrow})$ for all ideals $\mathfrak{x}'$ of $\mathfrak S'$. For that, let $\mathfrak{s}\in \phi_*(\mathfrak{x}'^{\uparrow}).$ This implies $\phi(\mathfrak{s})\in \mathfrak{x}'^{\uparrow},$ and that means $\mathfrak{x}'\subseteq \phi(\mathfrak{s}).$ Therefore, $\mathfrak{s}\in \phi\inv(\mathfrak{x}'^{\uparrow}).$ Since $\phi\inv(\mathfrak{x}'^{\uparrow})=\phi\inv(\mathfrak{x}'^{\uparrow})$, the other inclusion follows. If we take $\mathfrak{x}'$ as the trivial ideal $\mathfrak{o}'$ of $\mathfrak S'$, the above identity reduces to 
$\mathrm{Cl}(\phi_*(\sigma_{\mathfrak S'}))=\mathrm{Ker}(\phi)^{\uparrow},$ and hence  $\mathrm{Ker}(\phi)^{\uparrow}$ to be equal to $\sigma_{\mathfrak S}$ if and only if $\mathrm{Ker}(\phi)\subseteq \bigcap_{\mathfrak{s}\in \sigma_{\mathfrak S}}\mathfrak{s}.$ 
\end{proof}  

Note that for the spectrum of prime ideals, the inclusion condition in (\ref{den}) is replaced by an equality.  If $\phi$ is the quotient map $\mathfrak S\to \mathfrak{S}/\mathfrak{x}$, then we have the following

\begin{cor}
If $\sigma_{\mathfrak S}$ is a spectrum of $\mathfrak{S}$ satisfying the contraction property, then the Is\'{e}ki space $\sigma_{{\mathfrak{S}/\mathfrak{x}}}$ is homeomorphic to the closed subspace
$\mathfrak{x}^{\uparrow}$ of $\sigma_{\mathfrak S}$ for every $\mathfrak{x}$ of $\mathfrak{S}$.   
\end{cor} 

%\section*{Compliance with Ethical Standards The following are not applicable with respect to this article.\begin{enumerate}\item Disclosure of potential conflicts of interest.\item Research involving human participants and/or animals.\item Informed consent.\end{enumerate}

\end{document}